\documentclass[a4paper,11pt]{article}
\usepackage[T1]{fontenc}
\usepackage[latin1]{inputenc}
\usepackage{amssymb}
\usepackage[english]{babel}
\usepackage{amsthm}
\usepackage{amsmath}
\usepackage{amsfonts}
\usepackage{enumerate}

\usepackage{array,lmodern,graphicx}

\def\Re{\text{Re\ }}

\def\R{\mathbb{R}}
\def\C{\mathbb{C}}

\def\fra{\mathfrak{a}}
\def\frb{\mathfrak{b}}

\newtheorem{lem}{Lemma}[section]
\newtheorem{cor}[lem]{Corollary}
\newtheorem{prop}[lem]{Proposition}
\newtheorem{teo}[lem]{Theorem}
\newtheorem{df}[lem]{Definition}
\newtheorem{remark}[lem]{Remark}
\numberwithin{equation}{section}

\begin{document}

\title{Non-autonomous right and left multiplicative perturbations and maximal regularity}
\author{Mahdi Achache  and  El Maati Ouhabaz \thanks{\noindent Univ. Bordeaux, Institut de Math\'ematiques (IMB). CNRS UMR 5251. 351,  
Cours de la Lib\'eration 33405 Talence, France.
 Mahdi.Achache@math.u-bordeaux.fr, Elmaati.Ouhabaz@math.u-bordeaux.fr. 
 \newline Research partially supported by the ANR project  HAB: ANR-12-BS01-0013-02.}}
\date{}

\maketitle
\begin{abstract}
We consider the problem of  maximal regularity  for non-autonomous Cauchy problems
$$ u'(t) + B(t) A(t) u(t) + P(t) u(t) = f(t), \ u(0) = u_0$$
and 
$$ u'(t) + A(t) B(t) u(t) + P(t) u(t) =  f(t), \ u(0) = u_0.$$
In both cases, the time dependent operators $A(t)$ are associated with a family of sesquilinear forms  and the multiplicative left or right perturbations $B(t)$ as well as the additive perturbation $P(t)$ are families  of bounded operators on the considered Hilbert space. We prove maximal $L_p$-regularity results and  other regularity properties for the solutions of the previous problems under minimal regularity assumptions on the forms and perturbations. 
\vspace{.5cm}

\noindent \textbf{keywords:} Maximal regularity, non-autonomous evolution equations, multiplicative and additive perturbations.
\end{abstract}

\section{Introduction}\label{Sec1}
The present  paper deals with maximal $L_p$-regularity for non-autonomous evolution equations in the  setting of Hilbert spaces.  Before explaining our results we introduce some notations and assumptions.\\
 Let 
 $( \mathcal{ H}, (\cdot, \cdot ),\| \cdot \|)$ be a Hilbert space over $\R$ or $\C$. We consider another Hilbert space $\mathcal{V}$ which is densely and continuously embedded into $\mathcal{H}$. We denote by $\mathcal{V}'$ the (anti-) dual space of $\mathcal{V}$ so that 
$$\mathcal{ V} \hookrightarrow_{d} \mathcal{ H} \hookrightarrow_{d}  {\mathcal{V}}'.$$
We  denote by $\langle , \rangle$ the duality $\mathcal{V}$-$\mathcal{V}'$ and note that 
$\langle \psi, v \rangle = (\psi, v)$ if $\psi, v \in \mathcal{H}$.  
We consider a family of sesquilinear forms $$   \fra :[0,\tau]\times \mathcal{V} \times  \mathcal{V}  \rightarrow  \C $$ such that \begin{itemize}
 \item{} [H1]: $D(\fra(t))= \mathcal{V}$ (constant form domain),
 \item{}[H2]: $ | \fra(t,u,v)|\leq M \| u \|_\mathcal{ V} \| v \|_\mathcal{ V}$ (uniform boundedness),
\item{}[H3]: $\Re \fra(t,u,u)+\nu \| u \|^2\geq \delta  \| u \|_\mathcal{ V}^2 \ (\forall u\in{\mathcal V})\  \text{for some} \  \delta > 0$ \text{and some } $\nu \in  \R $ (uniform quasi-coercivity).
\end{itemize}
Here and throughout this paper, $\| \cdot \|_{\mathcal{V}}$ denotes the norm of $\mathcal{V}$. 

 To each form $\fra(t)$ we can associate two operators $A(t)$ and $\mathcal{A}(t)$ on $\mathcal{H}$ and $\mathcal{V}'$, respectively.  Recall that $u \in \mathcal{H}$ is in the domain $D(A(t))$
     if there exists $h\in  \mathcal{H}$ such that for all $v \in  \mathcal{V}$:
     $\fra(t,u,v) = (h, v)$.  We then set $A(t)u := h$. The operator $\mathcal{A}(t)$ is a bounded operator from $\mathcal{V}$ into $\mathcal{V}'$ such that $ \mathcal{A}(t)u = \fra(t, u, \cdot)$. The operator  $A(t)$ is the part of $\mathcal{A}(t)$ on $\mathcal{H}$. 
 It is a  classical fact that $-A(t)$ and $-\mathcal{A}(t)$ are both generators of holomorphic semigroups $(e^{-rA(t)})_{r\ge0}$
 and $(e^{-r\mathcal{A}(t)})_{r\ge0}$ on $\mathcal{H}$ and $\mathcal{V}'$, respectively. The semigroup $e^{-rA(t)}$ is the restriction
 of $e^{-r\mathcal{A}(t)}$ to $\mathcal{H}$. In addition,  $e^{-rA(t)}$ induces a holomorphic  semigroup on $\mathcal{V} $ (see, e.g.,  Ouhabaz \cite[Chapter 1]{Ouh05}). 

A well known result by J.L. Lions asserts that  the Cauchy problem
\begin{equation}\label{Cauchy1}
u'(t) + \mathcal{A}(t) u(t) = f(t), \ u(0) = u_0 \in \mathcal{H}
\end{equation}
has maximal $L_2$-regularity in $\mathcal{V}'$, that is, for every $f \in L_2(0, \tau; \mathcal{V}')$  there exists a unique
$u \in W^1_2(0, \tau; \mathcal{V}')$ which satisfies \eqref{Cauchy1} in the $L_2$-sense. The maximal regularity in $\mathcal{H}$ is however more interesting since when dealing with boundary value problems one cannot identify the boundary conditions if the Cauchy problem is considered in $\mathcal{V}'$. The maximal regularity in $\mathcal{H}$ is more difficult to prove. J.L. Lions has proved that this is the case for  initial data $u_0 \in D(A(0))$ under a quite restrictive regularity condition, namely $t \mapsto \fra(t,g,h)$ is $C^2$ (or $C^1$ if $u_0 = 0$). It was a question by him in 1961 (see \cite{Lions:book-PDE} p. 68) whether maximal $L_2$-regularity holds in general in $\mathcal{H}$. 

A lot of progress have been made in recent years on this problem. It was proved by Ouhabaz and Spina  \cite{OS} that maximal 
$L_p$-regularity holds in  $\mathcal{H}$ if  $t \mapsto \fra(t,g,h)$ is $C^\alpha$
for some $\alpha > 1/2$ (for all $g, h \in \mathcal{V}$). This result is however proved for the case  $u_0 = 0$ only. In Haak and Ouhabaz \cite{HO15}, it is proved that 
for $u_0 \in (\mathcal{H}, D(A(0)))_{1-\frac{1}{p}, p}$ and 
\begin{equation}\label{hou0}
| \fra(t,g,h) - \fra(s,g,h) | \le \omega( |t-s|) \| h \|_{\mathcal{V}} \| g \|_{\mathcal{V}}
\end{equation}
for some non-decreasing function $\omega$ such that
\begin{equation}\label{hou1}
\int_0^\tau \frac{\omega(t)}{t^{\frac{3}{2}}} dt < \infty \ \text{and} \ \int_0^\tau \left(\frac{\omega(t)}{t} \right)^p dt < \infty,
\end{equation}
then the Cauchy problem \eqref{Cauchy1} has maximal $L_p$-regularity in $\mathcal{H}$. The condition \eqref{hou1}  can be improved if  \eqref{hou0} holds with norms in some complex interpolation spaces (see Arendt and Monniaux \cite{AM} and Ouhabaz \cite{Ou15}). It was observed by Dier \cite{Dier14} that the answer to Lions' problem is negative in general. His example is based on non-symmetric forms for which  the Kato square root property $D(A(t))^{1/2}) = \mathcal{V}$ is not satisfied. Recently, Fackler \cite{Fac}   proved a negative answer to the maximal regularity problem for forms which are $C^\alpha$ for any $\alpha < 1/2$ (even symmetric ones). Let us also mention a recent positive result of Dier and Zacher \cite{DZ} on maximal $L_2$-regularity  in which the condition
\eqref{hou1} is replaced by a  norm in a Sobolev space  of order $ > \frac{1}{2}$.  For forms associated with divergence form elliptic operators, Auscher and Egert \cite{AE} proved that the order of this Sobolev space can be $\frac{1}{2}$. 

One of the aims of the present paper is to study the same problem for multiplicative perturbations. More precisely, we study
maximal $L_p$-regularity for 
\begin{equation}\label{Cauchy2}
u'(t) + B(t) A(t) u(t) + P(t) u(t) = f(t), \ u(0) = u_0
\end{equation}
and also for 
\begin{equation}\label{Cauchy3}
u'(t) + A(t) B(t) u(t) + P(t) u(t) =  f(t), \ u(0) = u_0,
\end{equation}
where $B(t)$ and $P(t)$ are bounded operators on $\mathcal{H}$ such that $\Re (B(t)^{-1} g,g) \ge \delta \| g \|^2$ for some $\delta > 0$ and all $g \in \mathcal{H}$. The left perturbation problem \eqref{Cauchy2} was already considered  by Arendt et al. \cite{ADLO} and the right perturbation one \eqref{Cauchy3} by Augner et al. \cite{JL15}. The two problems are motivated by applications to semi-linear evolution equations and boundary value problems. We extend the results in  \cite{ADLO} and \cite{JL15} in three   directions. The first one is to consider general forms which may not satisfy the Kato square root property, a condition which was used in an essential way  in the previous two papers. The second direction is to deal with maximal $L_p$-regularity, whereas in the mentioned papers only the maximal $L_2$-regularity is  considered. The third direction, which is our main motivation, is to assume less regularity on the forms
$\fra(t)$ with respect to $t$. In both papers \cite{ADLO} and \cite{JL15} it is assumed that $t \mapsto \fra(t,g,h)$ is Lipschitz continuous on $[0, \tau]$. In applications to elliptic operators with time dependent coefficients, the regularity assumption on the forms 
 reflects the regularity needed  for  coefficients with respect to $t$.
 
Our main results can be summarized as follows (see Theorems \ref{thmain} and \ref{thmright} for more general and precise statements). Suppose that for some $\beta, \gamma \in [0, 1]$, 
\begin{equation*}
|\fra(t, g,h)- \fra(s, g,h)| \leq \omega(|t-s|)\|g\|_{\mathcal{[H,V]}_{\beta}}  \|h\|_{\mathcal{[H,V]}_{\gamma}}, \ u, v \in \mathcal{V}
\end{equation*}
where $\omega: [0,\tau] \to  [0,\infty)$ is a non-decreasing function such that 
\begin{equation*}
\int_0^{\tau} \frac{\omega(t)}{t^{1+\frac{\gamma}{2}}}dt <\infty.
\end{equation*}
Suppose also that $t \mapsto B(t)$ is continuous on $[0, \tau]$ with values in $\mathcal{L(H)}$. Then the Cauchy problem \eqref{Cauchy2} has 
 maximal $L_p$-regularity in $\mathcal{H}$ for all $p \in (1, \infty)$ when $u_0 = 0$. 
If in addition, 
\begin{equation}\label{eq0000} 
\int_{0}^{\tau}\frac{\omega(t)^{p}}{t^{\frac{1}{2}(\beta+p \gamma)}}dt<\infty
\end{equation}
then \eqref{Cauchy2} 
 has maximal $L_p$-regularity in $\mathcal{H}$ provided $u_0 \in (\mathcal{H}, D(A(0)))_{1-\frac{1}{p},p}$. \\
 We also prove that if $\omega(t) \le C t^{\varepsilon}$  for some $\varepsilon > 0$ and $D(A(t)^{1/2}) = \mathcal{V}$ for all $t \in [0, \tau]$, then the solution $u \in C([0, \tau]; \mathcal{V})$ and $s \mapsto A(s)^{1/2} u(s) \in C([0, \tau]; \mathcal{H})$. 
 
 Concerning \eqref{Cauchy3}, we assume as in \cite{JL15} that $t \mapsto B(t)$ is Lipschitz continuous on $[0, \tau]$ with values in $\mathcal{L(H)}$. The assumptions on $\fra(t)$ are the same as above. The maximal $L_p$-regularity results we prove are the same as previously. We could also consider both left and right perturbations, see the end of Section \ref{sec4}. 
 
 We point out in passing that condition \eqref{eq0000}  is slightly better than the second condition in \eqref{hou1} which was assumed in \cite{HO15} and \cite{Ou15} (for the unperturbed problem).  In the natural case $\omega(t) \sim t^\alpha$,  one 
sees   immediately  that for large $p$, \eqref{hou1} requires larger $\alpha$ (and then more regularity) than \eqref{eq0000}. 
 
 In order to prove our results we  follow similar ideas as in \cite{HO15} and \cite{Ou15}. However, several modifications are needed in order to deal with  multiplicative perturbations. Also, at several places we  appeal to classical tools from harmonic analysis such as square function estimates or  H\"ormander type conditions for singular integral operators with vector-valued kernels. 
 
 Our results on maximal $L_p$-regularity could be applied to boundary values problems as well as to some semi-linear evolution equations. Such applications  have been already  considered in \cite{ADLO} and \cite{JL15}. The gain here is that we are able to 
assume less regularity with respect to the variable $t$.  We shall not write these applications  explicitly in this paper since the ideas are the same as in  \cite{ADLO} and \cite{JL15}, one has just to insert our new results on maximal regularity. The reader interested in applications of non-autonomous maximal regularity is referred to the previous articles and the references therein. \\
 
\noindent{\bf Notation.} We  denote by $\mathcal{L}(E,F)$ (or $\mathcal{L}(E)$) the space of bounded linear operators from $E$ to $F$ (from $E$ to $E$). The spaces $L_p(a,b; E)$ and $W^1_p(a,b; E)$ denote respectively the  Lebesgue and Sobolev spaces 
of function on $(a,b)$ with values in $E$. Recall that the norms of $\mathcal{H}$ and $\mathcal{V}$ are denoted by 
$\| \cdot \|$ and $\| \cdot \|_{\mathcal{V}}$. The scalar  product  of $\mathcal{H}$ is $(\cdot, \cdot)$.\\
Finally, we denote by $C$, $C'$ or $c...$  all inessential constants.  Their values may change from line to line. 

\section{The maximal regularity for the unperturbed problem}
Let $\mathcal{H}$ and $\mathcal{V}$ be as in the introduction. We consider  a family of sesquilinear forms
 $$\fra(t): \mathcal{V} \times  \mathcal{V} \to \C, \  t \in [0, \tau]$$
 which satisfy the classical assumptions [H1]-[H3]. We denote  again by $A(t)$ and $\mathcal{A}(t)$ the operators associated with 
 $\fra(t)$ on $\mathcal{H}$ and $\mathcal{V}'$, respectively. Note that by adding a positive constant to $\mathcal{A}(t)$ we may assume that
 [H3] holds with $\nu = 0$. 
  Therefore, there exists $w_0 \in [0, \frac{\pi}{2})$ such that 
\begin{equation}\label{eq3-1}
\fra(t,u,u)\in \overline{\Sigma(w_0)}, \ \ \forall t\in [0,\tau], u\in \mathcal{V}.
 \end{equation}
 Here
 $$\overline{\Sigma(w_0)} := \{z\in \mathbb{C}^{*} , | \arg (z)|\leq w_{0}\}.$$
 In \eqref{eq3-1}  we take $w_0$ to be the smallest possible value for  which the inclusion  holds. 

\begin{df}\label{df1}
Fix $u_0 \in {\mathcal H}$. We say that the problem
\begin{equation}\label{CP}
u'(t) + A(t) u(t) = f(t)\  (t \in [0, \tau]), \  u(0) = u_0
\end{equation}
has maximal $L_p$-regularity in $\mathcal{H}$ if for each $f\in L_p(0,\tau; \mathcal{H})$, there exists a unique $u\in W^{1}_{p}(0,\tau; \mathcal{H})$ such that $u(t)\in D(A(t))$ for almost all $t$ and  satisfies \eqref{CP} in the $L_p$-sense.
\end{df}
 We denote by $\mathcal{V}_\beta := [\mathcal{H}, \mathcal{V}]_\beta$   the classical complex interpolation space. Its usual norm is denoted $\| \cdot \|_{\mathcal{V}_\beta}$. We start with the following  result on maximal $L_p$-regularity of \eqref{CP}. 
 
\begin{teo}\label{thm3-1}
Suppose that the forms $(\fra(t))_{t\in [0,\tau]}$ satisfy the standing hypotheses [H1]-[H3]. Suppose that for some $\beta, \gamma \in [0, 1]$
\begin{equation}\label{eq3-2}
| \fra(t, u,v)- \fra(s, u,v)| \leq \omega(|t-s|)\|u\|_{\mathcal{V}_{\beta}}  \|v\|_{\mathcal{V}_{\gamma}}, \ u, v \in \mathcal{V},
\end{equation}
where $\omega: [0,\tau] \to  [0,\infty)$ is a  non-decreasing function such that 
$$\int_0^{\tau} \frac{w(t)}{t^{1+\frac{\gamma}{2}}}dt <\infty.$$
Then the Cauchy problem \eqref{CP} with $u_0 = 0$ has maximal $L_p$-regularity in $\mathcal{H}$ for all $p \in (1, \infty)$. \\
If in addition, 
\begin{equation}\label{eq3-3} 
\int_{0}^{\tau}\frac{w(t)^{p}}{t^{\frac{1}{2}(\beta+p \gamma)}}dt<\infty
\end{equation}
then  \eqref{CP} has maximal $L_p$-regularity in $\mathcal{H}$ for all $u_0\in (\mathcal{H}, D(A(0)))_{1-\frac{1}{p},p}$. Moreover there exists a positive constant C such that 
$$\| u \|_{W^1_p(0,\tau; \mathcal{H})} +\|A u\|_{L_p(0,\tau; \mathcal{H})} \leq C\left[\|f\|_{L_p(0,\tau; \mathcal{H})}+\|u_0\|_{( \mathcal{H},D(A(0)))_{1-\frac{1}{p},p}} \right].$$
Here, $( \mathcal{H},D(A(0)))_{1-\frac{1}{p},p}$ denotes the classical real-interpolation space and the constant $C$ depends only on the constants in [H1]-[H3].
\end{teo}
The first part of the theorem (i.e., the case $u_0 = 0$) was proved in \cite{HO15} when $\beta = \gamma = 1$ (and hence 
$[\mathcal{H}, \mathcal{V}]_\beta = [\mathcal{H}, \mathcal{V}]_\gamma = \mathcal{V}$).  The case with different 
values $\beta$ and $\gamma$ was proved  in \cite{Ou15}. See also  \cite{AM} for a related result. In order to treat the case of a non-trivial initial data 
$u_0 \in ( \mathcal{H}, D(A(0)))_{1-\frac{1}{p}, p}$, the assumption  required on $\omega$ in  \cite{HO15} is 
\begin{equation}\label{eq3-4}
 \int_0^\tau \left( \frac{\omega(t)}{t}\right)^p dt < \infty,
 \end{equation}
and in  \cite{Ou15},
\begin{equation}\label{eq3-5}
 \int_0^\tau \left( \frac{\omega(t)}{t^{\frac{\beta + \gamma}{2}}}\right)^p dt < \infty.
 \end{equation}
In the previous theorem we replace these conditions by the weaker condition \eqref{eq3-3}. The important example $\omega(t) = t^\alpha$ shows that 
\eqref{eq3-4} and \eqref{eq3-5}) require   a large  $\alpha$ (and hence  more regularity) in the case  $ p > 2$, whereas 
\eqref{eq3-3} does not require any additional regularity than $ \alpha > \frac{\gamma}{2}$ which  is already needed for  the first condition 
$$\int_0^{\tau} \frac{w(t)}{t^{1+\frac{\gamma}{2}}}dt <\infty.$$

\begin{proof} As explained above the  sole novelty here is the treatment of  the case $u_0 \in ( \mathcal{H}, D(A(0)))_{1-\frac{1}{p}, p}$ under the  condition \eqref{eq3-3}. 
Following  \cite{HO15} and \cite{Ou15}, we have to prove that 
\begin{equation}\label{eq3-6}
t \mapsto  A(t) e^{-tA(t)}u_0 \in L_p(0, \tau;  \mathcal{H}).
\end{equation}
Since we can assume without loss of generality that $A(0)$ is invertible, then  $u_0 \in ( \mathcal{H}, D(A(0)))_{1-\frac{1}{p}, p}$ is equivalent to (see \cite[Theorem~1.14]{Tri})
\begin{equation}\label{eq3-7}
t \mapsto    A(0) e^{-tA(0)}u_0 \in L_p(0, \tau; \mathcal{H}).
\end{equation}
For $g \in \mathcal{H}$  and a chosen contour  $\Gamma$ in the positive half-plane we write by  the holomorphic functional calculus
\begin{align*}
   & (A(t)  e^{-t A(t)} u_0 -  A(0)  e^{-t A(0)} u_0 , g)\\
 &=   \frac{1}{2\pi i} \int_\Gamma (z e^{-t z} \bigl[ (z I -A(t))^{-1} - (zI -  A(0))^{-1} \bigr] u_0, g) \, dz\\
 &=  \frac{1}{2\pi i}  \int_\Gamma  (z e^{-t z} \bigl[ \mathcal{A} (0)  - \mathcal{A} (t)\bigr] (z I -A(0))^{-1} u_0, (\overline{z} I -A(0)^*)^{-1} g) dz\\
 &=   \frac{1}{2\pi i}  \int_\Gamma  z e^{-t z} \bigl[\fra(0,  (z I -A(0))^{-1} u_0, (\overline{z} I -A(0)^*)^{-1} g) - \\
 & \hspace{5cm} \fra(t,  (z I -A(0))^{-1} u_0, {(z I -A(0))^{-1}}^* g) \bigr] \,dz.
\end{align*}
Hence by \eqref{eq3-2}, the modulus is bounded by
$$C \omega(t)  \int_0^\infty |z| e^{- c t |z|} \| (z I -A(0))^{-1} u_0\|_{\mathcal{V}_{\beta}}  \| {(z I -A(t))^{-1}}^*  g\|_{\mathcal{V}_{\gamma}}\, d|z|.$$
Note that by interpolation (see e.g. \cite{Ou15})
\begin{equation}\label{eq3-8}
\| (\overline{z}I - A(t)^*)^{-1} \|_{\mathcal{L(H, V_\gamma)}}  \le \frac{C}{| z |^{1 -\frac{\gamma}{2}}}.
\end{equation}
On the other hand for $f \in D(A(0))$, 
\begin{align*}
\delta \|(z I - A(0))^{-1} f \|^{2}_{\mathcal V} &\leq \Re (A(0) (z I - A(0))^{-1} f, (z I - A(0))^{-1}f)\\
&\le \| (z I - A(0))^{-1} A(0) f \| \| (z I - A(0))^{-1} f \| \\
&\le \frac{C}{|z|} \| A(0) f \|  \| (z I - A(0))^{-1} f \|_{\mathcal V}.
\end{align*}
The  embedding $\mathcal{V}  \hookrightarrow  \mathcal{ V}_\beta$ gives
$$ \| (z I - A(0))^{-1} \|_{\mathcal{L}(D(A(0)), \mathcal{V}_\beta)} \le \frac{C}{| z |}.$$
Hence, by \eqref{eq3-8} and interpolation
\begin{equation}\label{eq3-9}
\| (z I - A(0))^{-1} \|_{\mathcal{L}( (\mathcal{H}, D(A(0)))_{1-\frac{1}{p}, p}, \mathcal{V}_\beta)} \le \frac{C}{| z |^{1-\frac{\beta}{2p}}}.
\end{equation}
Using these estimates we obtain 
\begin{align*}
&\, |(A(t)  e^{-t A(t)} u_0 -  A(0)  e^{-t A(0)} u_0 , g) | \\
&\, \le C \omega(t) \int_0^\infty  \frac{e^{- c t |z|}}{ |z|^{1 - \frac{1}{2}(\gamma + \frac{\beta}{p})}}\, d |z| \|g\| 
\|u_0 \|_{( \mathcal{H}, D(A(0)))_{1-\frac{1}{p}, p}} \\
&\, \le C' \frac{\omega (t)}{ t^{\frac{1}{2}(\gamma + \frac{\beta}{p})}} \|g\| \|u_0 \|_{( \mathcal{H}, D(A(0)))_{1-\frac{1}{p}, p}}.
\end{align*}
Hence, $t \mapsto  A(t) e^{-tA(t)}u_0 \in L_p(0, \tau, \mathcal{H})$ for $u_0 \in ( \mathcal{H}, D(A(0)))_{1-\frac{1}{p}, p}$ if  $\omega(t)$ satisfies 
\eqref{eq3-3}.
\end{proof}

\section{Maximal regularity for left perturbations}\label{sec3}
This section is devoted to the main subject of this paper in which we  are interested in maximal regularity for operators 
$B(t)A(t)$ for a wide class of operators $B(t)$ and $A(t)$. We will consider in another section  the same problem for right multiplicative perturbations $A(t)B(t)$.

\subsection{Single left multiplicative pertubation-Resolvent estimates}\label{Sec2}
Let $\mathcal{H}$ and $\mathcal{V}$ be as above. We denote again by $\| \cdot \|$ and $\| \cdot \|_\mathcal{V}$ their associated norms, respectively. \\
Let $\fra: \mathcal{V} \times  \mathcal{V}  \rightarrow  \mathbb{C}$ be a closed, coercive and continuous sesquilinear form.  We denote by 
$A$ and $\mathcal{ A}$ its associated operators on $\mathcal{H}$ and $ {\mathcal{V}}'$, respectively.\\
Let  $\frb: \mathcal{ H}\times \mathcal{ H} \rightarrow \mathbb{C}$ be a bounded sesquilinear form. We assume that $\frb$ is coercive, that is 
there exists a constant $\delta > 0$  such that
\begin{equation}\label{eq2-1}
\Re  \frb(u,u)\geq \delta \| u \|^{2}, \  u \in  \mathcal{H}.  
\end{equation}
There exists a unique bounded operator associated with $\frb$. We  denote temporarily this operator by  $\mathcal{C}$.
 Note that by coercivity, it is obvious that $\mathcal{C}$ is invertible on $\mathcal{H}$.

 Now we introduce another operator $A_\frb$ which we call {\it the operator associated  with $\fra$ with respect to $\frb$}. It is defined as follows  
$$ D(A_{\frb})= \{u\in \mathcal{ V},\exists v\ \mathcal{ H}: \fra(u,\phi)= \frb(v,\phi) \ \forall \phi \in \mathcal{ V} \}, \ A_{\frb}u:= v.$$ 
The difference with $A$ is that we take the form $\frb$ instead of the scalar product of $\mathcal{H}$  in the equality $\fra(u,\phi)= \frb(v,\phi).$
The operator $A_{\frb}$ is well defined. Indeed, if 
$\frb(v_{1},\phi) = \frb(v_{2},\phi)$ for all $\phi \in \mathcal{ V}$ then by density this equality holds for all $\phi \in \mathcal{H}$. Therefore, 
taking  $\phi=v_{2}-v_{1} $ and using  \eqref{eq2-1},  we obtain  $v_{2}=v_{1}$. 
\begin{prop}
Let $B :=\mathcal{C}^{-1}$. Then $A_{\frb}=BA$ with domain $D(A_\frb) = D(A)$.  
\end{prop}
\begin{proof}
Let $ u \in D(A_\frb)$ and $v= A_\frb u$. Then 
$$\fra(u,\phi)= \frb(v,\phi)= (\mathcal{C} v,\phi) \ \forall \phi \in \mathcal{V}.$$
Thus,  $u\in D(A) $ and $Au=\mathcal{C} v = B^{-1}v$. This gives, $u\in D(A) $  and $A_\frb u = v = BA u$.\\
For the converse, we write for $u \in D(A)$ 	and $\phi \in \mathcal{V}$
$$\fra(u,\phi) = (Au, \phi) = (\mathcal{C} BAu, \phi) = \frb( BAu, \phi).$$
This gives $u \in D(A_\frb)$ and $BA u = A_\frb u$. 
\end{proof}

It is obvious that $BA$ is a closed operator on $\mathcal{H}$.
In order to continue we assume that $\fra$ is coercive (i.e., it satisfies [H3] with $\nu = 0$) and define $w_{0}$ and $w_{1}$ to be the angles of the numerical ranges of $A$ and $B$, respectively. That is 
$$  (Au,u)\in \overline{\Sigma (w_{0})}  := \{z\in \mathbb{C}^{*} ,| \arg (z)|\leq w_{0}\}  $$
and
$$\frb(u,u)=(B^{-1}u,u) \in \overline{\Sigma (w_{1})}$$
where $w_{0}$ and $ w_{1}$ are the smallest possible values for which  these two properties 
hold for all $u\in \mathcal{V}$. 
Note that $w_{0}, w_{1} \in[0,\frac{\pi}{2})$ because of the coercivity property.  

\begin{prop} \label{prop2-1}
For all  $\lambda \notin \overline{\Sigma(w_{0}+w_{1})}$,  the operator  $\lambda I-BA$ is invertible on $\mathcal{H}$ and 
$$\|(\lambda I-BA)^{-1} \|_{\mathcal{L(H)}} \leq \frac{\delta^{-1}  \| B^{-1} \|_{\mathcal{L(H)}}  }{dist(\lambda,\overline{\Sigma(w_{0}+w_{1})}) }.$$
\end{prop}
  
\begin{proof}
Let $u\in D(A)$. We write 
\begin{align*}
 \|(\lambda I-BA) u \| \|u\| &=\|B(\lambda B^{-1} -A) u  \| \|u\|  \\
&\geq \frac{1}{\|B^{-1}\|_{\mathcal{L(H)}}}\|(\lambda B^{-1} I-A)u \|\|u\| \\
&\geq \frac{1}{\|B^{-1}\|_{\mathcal{L(H)}}}|(\lambda B^{-1} u-Au,u)|  \\
&=\frac{|(B^{-1}u,u)|}{\|B^{-1}\|_{\mathcal{L(H)}}}|\lambda-\frac{(Au,u)}{(B^{-1}u,u)} |.
\end{align*}
Since $ \frac{(Au,u)}{(B^{-1}u,u)} = \frac{\fra(u,u)}{\frb(u,u)} \in \overline{\Sigma(w_{0}+w_{1})} $ it follows that  
$$ |(\lambda-\frac{(Au,u)}{(B^{-1}u,u)} | \geq dist(\lambda,\overline{\Sigma(w_{0}+w_{1})}).$$
On the other hand, by \eqref{eq2-1}, $|(B^{-1}u,u)|\geq \delta   \|u\|^{2}$ and so 
$$\|(\lambda I-BA)u \| \|u\|\geq \frac{\delta}{\|B^{-1}\|_{\mathcal{L(H)}}} \|u\|^{2} dist(\lambda,\overline{\Sigma(w_{0}+w_{1})}).$$
Hence, 
\begin{equation}\label{eq2-2}
 \|(\lambda I-BA) u \| \geq  \frac{\delta}{\|B^{-1}\|_{\mathcal{L(H)}} } \|u\| dist(\lambda,\overline{\Sigma(w_{0}+w_{1})})\ \ \forall u\in D(A). 
 \end{equation}
This implies  that $\lambda I-BA$ is injective and has closed range for $ \lambda \not\in \overline{\Sigma(w_{0}+w_{1})}$.\\
In order to prove that $ \lambda I-BA$ is invertible it remains  to prove that it has dense range.
By duality, one has to  prove that the adjoint is injective.The adjoint operator is $\overline{\lambda}I-A^{*}B^{*}$.
We write 
$$\overline{\lambda}I-A^{*}B^{*}= (\overline{\lambda} {B^{*}}^{-1}-A^{*})B^{*}.$$
The previous  arguments show that $\lambda B^{-1}-A$ is injective.
This also applies to $\lambda {B^{*}}^{-1}-A^{*} $. 
Since $B^*$ is invertible, we obtain $\overline{\lambda}I-A^{*}B^{*}$ is injective and hence $\lambda I-BA $ is invertible .
Now \eqref{eq2-2}  gives 
$$ \|(\lambda I-BA)^{-1} \| \leq \frac{\|B^{-1}\|_{\mathcal{L(H)}}}{\delta.dist(\lambda,\overline{\Sigma(w_{0}+w_{1})})}$$
for all $\lambda \not\in \overline{\Sigma(w_{0}+w_{1})}).$
\end{proof}  

\begin{cor}\label{cor2-2}
Suppose that $w_0 +w_1 < \frac{\pi}{2}$.Then $-BA$ is the generator of a bounded holomorphic semigroup on $\mathcal{H}$.
\end{cor}

\begin{proof}
By Proposition \ref{prop2-1},
$$\|(\lambda I-BA)^{-1}\| \leq \frac{c}{|\lambda|}, \ \ \forall  \lambda \not\in \overline{\Sigma(w_{0}+w_{1})}) $$
In other words, $\lambda I+BA $ is invertible for $\lambda \in \Sigma (\pi-(w_0+w_1)) $ and :
$$\|(\lambda I+BA)^{-1}\| \leq \frac{c}{|\lambda|}, \ \ \forall  \lambda \in {\Sigma(\pi-(w_{0}+w_{1})}). $$
It is a classical fact that the latter estimate implies that $-BA$ generates a bounded holomorphic semigroup of angle $\frac{\pi}{2}-(w_0+w_1)$.
\end{proof}

Obviously,  one cannot remove  the assumption $w_0 +w_1 < \frac{\pi}{2}$ in the previous result. Indeed, let $A = - e^{i \frac{\pi}{3}} \Delta$ on $L_2(\R^d)$ and $B$  be the multiplication by $e^{i \frac{\pi}{3}}$. Then  $-BA = e^{i \frac{2\pi}{3}} \Delta$ is not a generator of a $C_0$-semigroup. 


\subsection{Single pertubation-Maximal regularity }\label{subsec3-1}
Let $(\fra(t))_{t\in [0,\tau]}, A(t), \mathcal{A}(t)$ and $\frb$ be as in the previous sub-section. We assume that [H3] holds with $\nu = 0$. 
In particular, \eqref{eq3-1} holds.  We also have 
\begin{equation}\label{eq3-2-1}
\frb(u,u)\in \overline{\Sigma(w_1)}
\end{equation}
 for some $w_1\in [0,\frac{\pi}{2})$ by coercivity of $\frb$. \\
We make the assumption $w_0 +w_1 <\frac{\pi}{2}$.  By Corollary \ref{cor2-2}, for each $t \in [0, \tau]$, the operator 
$-BA(t)$ generates a holomorphic semigroup $(e^{-s BA(t)})_{s\ge0}$ on $\mathcal{H}$. 

Our aim in this section is to prove maximal regularity in $\mathcal{H}$ for the Cauchy problem associated with 
$BA(t)$, $t\in [0,\tau]$.  The definition of maximal $L_p$-regularity in this context is the same as in Definition \ref{df1}.\\
Set
$$R(\lambda, BA(t)) := (\lambda I + BA(t))^{-1}$$
for $\lambda \in \rho(-BA(t))$. 

\begin{prop}\label{prop3-1}
Assume that $w_0 +w_1 <\frac{\pi}{2}$. Then
\begin{enumerate}
\item[1-]$\|(\lambda B^{-1}+A(t))^{-1}\|_{\mathcal{L( H)}}\leq \frac{C}{|\lambda|+1}, \  \lambda \in {\Sigma(\pi - (w_{0}+w_{1})})$, 
\item[2-]$ \|R(\lambda,BA(t))B\|_{\mathcal{L(V', H)}}\leq \frac{C}{(|\lambda|+1)^{\frac{1}{2}}}, \  \lambda \in {\Sigma(\pi - (w_{0}+w_{1})})$,                                        
\item[3-]$\|e^{-(t-s)BA(t)}B\|_{\mathcal{L(V', H)}}\leq \frac{C}{(t-s)^{\frac{1}{2}}}$,
\item[4-]$\|e^{-(t-s)BA(t)}B\|_{\mathcal{L(V', V)}}\leq \frac{C}{(t-s)}$.
\end{enumerate}
The constant $C$ is independent of $t$ and $\lambda$. 
\end{prop}
\begin{proof}
We have  $(\lambda B^{-1}+A(t))^{-1}=(\lambda  +BA(t))^{-1}B$, then we  obtain assertion 1- from Proposition \ref{prop2-1}.  \\
Note that 
\begin{equation}\label{eq3-2-2}
(\lambda B^{-1}+A(t))^{-1}= (\lambda +A(t))^{-1}+(\lambda B^{-1}+A(t))^{-1}(\lambda(-B^{-1}+I))(\lambda+A(t))^{-1}.
\end{equation}
 Then
 \begin{align*}
     \|R(\lambda,BA(t)) B\|_{\mathcal{L(V', H)}} &= \|(\lambda B^{-1}+A(t))^{-1}\|_{\mathcal{L(V', H)}}\\
   & \le  \|(\lambda I +A(t))^{-1}\|_{\mathcal{L(V', H)}} +  \\
   & \hspace{-2cm}  \|(\lambda B^{-1}+A(t))^{-1}(\lambda(-B^{-1}+I))\|_{\mathcal{L(H)}}\|(\lambda+A(t))^{-1}\|_{\mathcal{L(V', H)}}.
 \end{align*}
Since    $$ \|(\lambda I +A(t))^{-1}\|_{\mathcal{L(V', H)}}  \le \frac{C}{(|\lambda|+1)^{\frac{1}{2}}}$$
(see e.g. \cite{HO15}), we obtain                                
 $$ \|R(\lambda,BA(t))B\|_{\mathcal{L(V', H)}} \le \frac{C}{(|\lambda|+1)^{\frac{1}{2}}},$$   
which proves  assertion 2. \\                                     
 Now we  choose an appropriate contour $\Gamma = \partial\Sigma(\theta)$ with $\theta <\frac{\pi}{2}$ and  write by the  functional calculus 
   $$e^{-(t-s)BA(t)}B= \frac{1}{2\pi i} \int_\Gamma e^{-(t-s)\lambda}(\lambda  - BA(t))^{-1}B d\lambda. $$ 
Then
\begin{align*}
\|e^{-(t-s)BA(t)}B\|_{\mathcal{L(V', H)}} &\leq  \frac{1}{2\pi} \int_0^{\infty} e^{-(t-s)\Re{\lambda}}\| (\lambda  - BA(t))^{-1}B\|_{\mathcal{L(V', H)}}d|\lambda| \\
&\leq  C \int_0^{\infty} e^{-(t-s)\Re{\lambda}} \frac{1}{{(|\lambda|+1)}^\frac{1}{2}}d|\lambda| \\
&\leq \frac{C'}{(t-s)^{\frac{1}{2}}}.
\end{align*}  
In order to prove assertion 4- we write
$$\|e^{-(t-s)BA(t)}B\|_{\mathcal{L(V', V)}} \le \|e^{-\frac{(t-s)}{2}BA(t)}B B^{-1}\|_{\mathcal{L(H, V)}}  \|e^{-\frac{(t-s)}{2}BA(t)}B \|_{\mathcal{L(V', H)}}$$
and 
$$\|e^{-\frac{(t-s)}{2}BA(t)}B B^{-1}\|_{\mathcal{L(H, V)}} \le  \| B^{-1} \|_{\mathcal{L(H)}}  \| e^{-\frac{(t-s)}{2}BA(t)}B \|_{\mathcal{L(H, V)}}. $$
We use the equality 
$$ (\lambda B^{-1} + A(t))^{-1} = (\lambda I + A(t))^{-1} + (\lambda I + A(t))^{-1} \lambda(I - B^{-1}) (\lambda B^{-1}+ A(t))^{-1}$$
in place of \eqref{eq3-2-2} to estimate $\| R(\lambda, BA(t))B \|_{\mathcal{L(H,V)}}$ and then argue as previously.
\end{proof}

Now, let $P(t) \in \mathcal{L(H)} $ such that $t \mapsto P(t)$ is strongly measurable and 
\begin{equation}\label{P1}
 \| P(t) \|_{\mathcal{L(H)}} \le M, \ t \in [0, \tau]
 \end{equation}
for some constant $M$. We  consider the Cauchy problem 
\begin{equation}\label{CPBP}
u'(t) + BA(t)  u(t)  + P(t) u(t) = f(t),  \  \ u(0) = u_0.
\end{equation}
Recall that $B^{-1}$ is the operator associated with $\frb$. We are interested in maximal regularity of \eqref{CPBP}. As explained at the beginning of the proof of the next proposition, we may  assume without loss of generality that  the forms $\fra(t)$ are coercive and hence  \eqref{eq3-1} is satisfied for some $w_0 \in [0, \frac{\pi}{2})$. 
\begin{prop}\label{prop3-2-1}
Suppose that the forms $(\fra(t))_{t\in [0, \tau]}$ satisfy [H1]-[H3], the form $\frb$ satisfies  \eqref{eq2-1} and $w_0 + w_1 < \frac{\pi}{2}$.  Suppose that for some $\beta, \gamma \in [0, 1]$
$$|\fra(t, u,v)- \fra(s, u,v)| \leq \omega(|t-s|)\|u\|_{\mathcal{V}_{\beta}}  \|v\|_{\mathcal{V}_{\gamma}}, \ u, v \in \mathcal{V}$$
where $\omega: [0,\tau] \to  [0,\infty)$ is a non-decreasing function such that :
$$\int_0^{\tau} \frac{w(t)}{t^{1+\frac{\gamma}{2}}}dt <\infty.$$
Then the Cauchy problem \eqref{CPBP} with $u_0 = 0$ has maximal $L_p$-regularity in $\mathcal{H}$ for all $p \in (1, \infty)$. \\
If in addition, 
\begin{equation}\label{eq3-3-3} 
\int_{0}^{\tau}\frac{\omega(t)^{p}}{t^{\frac{1}{2}(\beta+p \gamma)}}dt<\infty
\end{equation}
then \eqref{CPBP} has maximal $L_p$-regularity for all $u_0 \in (\mathcal{H}, D(A(0)))_{1-\frac{1}{p},p}$. Moreover there exists a positive constant C such that :
$$\| u \|_{W^1_p(0,\tau; \mathcal{H})} +\|A u\|_{L_p(0,\tau; \mathcal{H})} \leq C\left[ \|f\|_{L_p(0,\tau; \mathcal{H})}+\|u_0\|_{( \mathcal{H}, D(A(0)))_{1-\frac{1}{p},p}} \right].$$
Here  $C$ depends only on the constants in [H1]-[H3], $\| B \|_{\mathcal{L(H)}}$, $\| B^{-1} \|_{\mathcal{L(H)}}$ and $M$ in  \eqref{P1}.
\end{prop}

\begin{proof}
Firstly,  we note that for $c \in \R$,  \eqref{CPBP} has maximal $L_p$-regularity {\it if and only if} the Cauchy problem
$$ v'(t) + (B A(t) + P(t) + c I) v(t) =  e^{-ct}f(t), \ \  v(0) = u_0$$
has maximal $L_p$-regularity. The reason is that $v(t) = u(t) e^{-ct}$ and it is clear that  $u \in W^1_p(0, \tau; \mathcal{H})$ {\it if and only if} $v \in W^1_p(0, \tau; \mathcal{H})$.\\
Thus,  by adding a large constant $c$ we may assume that [H3] holds with $\nu = 0$ and $BA(t) + P(t) $ 
is invertible for each $t \in [0, \tau]$.\\ 
Note that  $BA(t)= A(t)_{\frb}$ is the operator associated with the form ${\fra}(t)$ with respect to $\frb$ (see Section \ref{Sec2}). This allows us to use the same strategy of proof as for Theorem \ref{thm3-1} (cf.  \cite{Ou15} or \cite{HO15} in the case $\beta= \gamma = 1$). 

Set $v(s) := e^{-(t-s)BA(t)}u(s)$. Writing $ v(t) - v(0) = \int_0^t v'(s) ds$ we obtain
\begin{align*}
A(t)u(t)=A(t)e^{-t BA(t)}u_{0}&+A(t)\int_{0}^{t} { e^{-(t-l) BA(t)}B(\mathcal{A}(t)-\mathcal{A}(l))u(l)dl}\\&+A(t)\int_{0}^{t} { e^{-(t-l) BA(t)}(-P(l))u(l)dl}\\   &+A(t)\int_{0}^{t} { e^{-(t-l) BA(t)}f(l)}dl.
\end{align*}
Note that by Proposition \ref{prop3-1}, the term $e^{-(t-l) BA(t)}B(\mathcal{A}(t)-\mathcal{A}(l))u(l)$ is well defined.\\
We first prove the proposition in the case $u_0 = 0$. We  define
$$(Lf)(t) := A(t)\int_{0}^{t} { e^{-(t-l)BA(t)}f(l)}dl.$$
Following \cite{HO15} the operator $L$ is a pseudo-differential operator with the vector-valued symbol $\sigma(t, \xi)$ given by
\[
\sigma(t, \xi) := 
\left\{
  \begin{array}{lcl}
     A(0)(i\xi+ B(0)A(0))^{-1} & \text{if} &  t<0  \\
     A(t)(i\xi+B(t)A(t))^{-1} & \text{if} &  0\le t \le \tau \\
     A(\tau)(i\xi+ B(\tau)A(\tau))^{-1} & \text{if} &  t > \tau. \\
  \end{array}\right.
\]
 Then we use Proposition \ref{prop2-1} and argue as in the proof of Lemmas 10 and 11 in \cite{HO15} to prove  the boundedness on $L_p(0, \tau; \mathcal{H})$, $1 < p < \infty$,  of the operator $L$. \\
We continue as in \cite{HO15} and \cite{Ou15}. We set 
$$(Sg)(t) :=A(t)\int_{0}^{t} { e^{-(t-l) BA(t)}(P(l))A(l)^{-1}g(l)dl}.$$
By the boundedness of the operator $L$ on $L_p(0,\tau; \mathcal{H})$, 
$$\|Sg\|_{L_p(0,\tau; \mathcal{H})} \leq C \|A^{-1}g\|_{L_p(0,\tau; \mathcal{H})}.$$
We define
$$(Qg)(t) :=A(t)\int_{0}^{t} { e^{-(t-l) BA(t)}B(\mathcal{A}(t)-\mathcal{A}(l))A(l)^{-1}g(l)dl}.$$
Then, arguing as in \cite{HO15} or \cite{Ou15} we obtain easily from Proposition \ref{prop3-1}
$$\|(Qg)(t)\| \leq \int_{0}^{t} \frac{w(|t-l|)}{(t-l)^{1+\frac{\gamma}{2}}} \|A^{-1}(l)g(l)\|_{\mathcal{V}} dl.
$$
Thus,
$$\|Qg\|_{L_p(0,\tau; \mathcal{H})} \leq C \int_{0}^{\tau} \frac{w(t)}{t^{1+\frac{\gamma}{2}}}dt \|A^{-1}g\|_{L_p(0,\tau;\mathcal{V})}. $$
From these estimates, we see that by replacing $A(t)$ by  $A(t)+ c I$ 
for  $c$ large enough we obtain
  $$\| S\|_{\mathcal{L} (L_p(0,\tau; \mathcal{H}))} <  \frac{1}{4} \  \text{and} \ \|Q\|_{\mathcal{L} (L_p(0,\tau;\mathcal{H}))} < \frac{1}{4}.$$
 In particular,  $ I- (S+Q)$ is invertible. Since  
 $$(Au)(t)=(I-(S+Q))^{-1}(L(f))(t)$$
we obtain $Au \in L_p(0,\tau; \mathcal{H})$ and hence $u \in W^1_p(0, \tau; \mathcal{H})$. This proves maximal $L_p$-regularity.  
 
In order to treat  the case $u_0 \not= 0$ we need to estimate the difference of the resolvents, i.e., 
$\|R(\lambda,A(t)_{b})-R(\lambda,A(s)_{b})\|_{\mathcal{L( H)}}$ in terms of $\omega(|t-s|)$. \\
Let $f, g\in \mathcal{H}$ and $\lambda \in \Sigma(\pi-(w_0+w_1))$. We write 
\begin{align*}
&( [R(\lambda,A(t)_{\frb})-R(\lambda,A(s)_{\frb})]f, g) \\
&\hspace{3cm} = - ( [R(\lambda, BA(t))B ({\mathcal{ A}}(t)- {\mathcal{ A}}(s))R(\lambda, B{A}(t))]f,g).
\end{align*}
Note that the RHS is well defined since $R(\lambda, B{A}(t))B$ is a  bounded operator from $\mathcal{V}'$ to $\mathcal{V}$ (cf. Proposition \ref{prop3-1}). 
Therefore,
\begin{align*}
& \hspace{-1.5cm}( [R(\lambda,{A}(t)_{b})-R(\lambda,{A}(s)_{b})]f,g)\\  
&=\langle ({\mathcal{ A}}(t)- {\mathcal{ A}}(s))R(\lambda, B{A}(t))f,B^{*}R(\overline{\lambda},B{A}(t))^* g\rangle \\
&= {\fra}(s,R(\lambda, B{A}(s))f,(\overline{\lambda}{B^{*}}^{-1} + {A}(t)^{*})^{-1}g) \\
& \hspace{3cm} - {\fra}(t,R(\lambda, B{A}(s))f,(\overline{\lambda}{B^{*}}^{-1} + {A}(t)^{*})^{-1}g).
\end{align*}
Hence the modulus is bounded by 
$$ \omega(|t-s|)\|R(\lambda,B{A}(s))f\|_{\mathcal{V}_\beta} \|{\overline{\lambda}}{B^{*}}^{-1} + {A}(t)^{*})^{-1} g \|_{\mathcal{V}_\gamma}. $$
Let ${w_0}$ be the common angle for  the numerical range of ${\fra}(t)$. By  Proposition \ref{prop2-1} we have for all $\lambda \notin \overline{\Sigma({w_0} + w_1)}$
\begin{align*}
\delta \|R(\lambda,{A}(s)_{\frb})f\|^{2}_{\mathcal V} &\leq \Re \fra(s,R(\lambda,{A}(s)_\frb)f,R(\lambda,{A}(s)_\frb)f)\\
&=\Re({A}(s)R(\lambda,{A}(s)_{\frb})f,R(\lambda,{A}(s)_{\frb})f)\\
&=\Re(B{A}(s)R(\lambda, {A}(s)_{b})f,(B^{-1})^{*}R(\lambda,{A}(s)_{\frb})f)\\
&\leq \frac{C}{|\lambda|}\|f\|^{2} .
\end{align*}
Hence, by interpolation
\begin{equation}\label{eq3-3-4}
\|R(\lambda,{A}(s)_{\frb})f\|^{2}_{\mathcal{V}_\beta } \le  \frac{C}{ | \lambda|^{1- \frac{\beta}{2}}}.
\end{equation}
Putting together the previous estimates yields
 $$| \frb( [R(\lambda, B{A}(t) )-R(\lambda, B{A}(s))]f,g)|\leq C \frac{\omega(|t-s|)}{|\lambda|^{2 - \frac{\beta + \gamma}{2}}} \|f\| \|g\|.$$
 This shows 
 $$ \|R(\lambda, {A}(t)_\frb)-R(\lambda,{A}(s)_\frb) \|_{\mathcal{L(H)}}\leq C \frac{\omega(|t-s|)}{|\lambda|^{2 - \frac{\beta + \gamma}{2}}}. $$   
This is the  estimate we need in order to obtain the proposition when $u_0 \in (\mathcal{H}, D(A(0))_{1- \frac{1}{p}, p}$ (see \cite{HO15} or \cite{Ou15} for the details). 
\end{proof}

\subsection{Time dependent perturbations-Maximal regularity }
Let $\fra(t), A(t), \mathcal{V}$ and $\mathcal{H}$  be as above and suppose again that  the standard  assumptions [H1]-[H3] are satisfied.  Let  $(B(t))_{t\in [0,\tau]}$ be a family  of bounded invertible operators on $\mathcal{H}$. 
We assume that there exist constants $\delta > 0$ and $M > 0$ independent of $t$ such that   
\begin{equation}\label{eq3-4-1}
\Re ( B(t)^{-1} u, u ) \ge \delta \| u \|_{\mathcal{H}}^2  \ \ \forall u \in \mathcal{H},
\end{equation}
and
\begin{equation}\label{eq3-4-2}
 \| B(t)^{-1} \|_{\mathcal{L(H)}} \le M.  
\end{equation}
Let $(P(t))_{t \in [0, \tau]}$ be a family of bounded operators on $\mathcal{H}$. We assume that  
\begin{equation}\label{eq3-4-3}
\| P(t) \|_{\mathcal{L(H)}} \le M.
\end{equation}

As a consequence of \eqref{eq3-4-1} and \eqref{eq3-4-2} the numerical range of $B(t)^{-1}$ is contained in a sector of angle $w_1$ for some 
$w_1 \in [0, \frac{\pi}{2})$, independent of $t$. Note that \eqref{eq3-4-1} implies that
$$ \| B(t)^{-1} u \| \ge \delta \| u \|$$
and hence
\begin{equation}\label{eq3-4-4}
 \| B(t)  \|_{\mathcal{L(H)}} \le \frac{1}{\delta}.  
\end{equation}
We denote as previously  by $w_0$ the common angle of the numerical range of forms $\fra(t), t \in [0, \tau]$. We assume again that 
\begin{equation}\label{eq3-4-5}
w_0 + w_1 < \frac{\pi}{2}.
\end{equation}

The following is our main result. 
\begin{teo}\label{thmain}
Suppose that $(\fra(t))_t$ satisfies [H1]-[H3]. Let $B(t)$ and $P(t)$ be bounded operators which satisfy \eqref{eq3-4-1}-\eqref{eq3-4-3} and 
\eqref{eq3-4-5}.  Suppose in addition that $t \mapsto B(t)$ is   continuous on $[0,\tau]$ with values in $\mathcal{L(H)}$.
Suppose that for some $\beta, \gamma \in [0, 1]$
\begin{equation}\label{eqeq0}
|\fra(t, u,v)- \fra(s, u,v)| \leq \omega(|t-s|)\|u\|_{\mathcal{V}_{\beta}}  \|v\|_{\mathcal{V}_{\gamma}}, \ u, v \in \mathcal{V}
\end{equation}
where $\omega: [0,\tau] \to  [0,\infty)$ is a non-decreasing function such that :
\begin{equation}\label{eqmain}
\int_0^{\tau} \frac{\omega(t)}{t^{1+\frac{\gamma}{2}}}dt <\infty.
\end{equation}
Then the Cauchy problem 
\begin{equation}\label{CPP0}
u'(t) + B(t)A(t)u(t)  + P(t)u(t) = f(t), \ \ u(0) = 0
\end{equation}
 has maximal $L_p$-regularity in $\mathcal{H}$ for all $p \in (1, \infty)$. \\
If in addition, 
\begin{equation}\label{eq3-4-6} 
\int_{0}^{\tau}\frac{\omega(t)^{p}}{t^{\frac{1}{2}(\beta+p \gamma)}}dt<\infty
\end{equation}
then 
\begin{equation}\label{CPP}
u'(t) + B(t)A(t)u(t)  + P(t)u(t) = f(t),  \ \ u(0) = u_0
\end{equation}
 has maximal $L_p$-regularity in $\mathcal{H}$ provided $u_0 \in (\mathcal{H}, D(A(0)))_{1-\frac{1}{p},p}$. 
 Moreover there exists a positive constant C such that :
\begin{equation}\label{apriori}
\| u \|_{W^1_p(0,\tau; \mathcal{H})} +\|BAu\|_{L_p(0,\tau; \mathcal{H})} \leq C \left[\|f\|_{L_p(0,\tau; \mathcal{H})}+\|u_0\|_{( \mathcal{H}, D(A(0)))_{1-\frac{1}{p}, p}} \right].
\end{equation}
The constant $C$ depends only on the constants in [H1]-[H3], $\delta $ and $M$ in \eqref{eq3-4-1}-\eqref{eq3-4-3}.
\end{teo}

\noindent{\bf Remark.} As we shall see in the proof, the regularity assumption on $B(t)$ can be weakened considerably. Indeed,  
continuity at finite number of appropriate points  is sufficient.  \\

Before starting the proof, let us define the maximal regularity space
$$MR(p,\mathcal{H}) :=\{u\in W^1_p(0,\tau; \mathcal{H}): u(t)\in D(A(t))\  \text{a.e.},\   A(.)u(.)\in L_p(0,\tau; \mathcal{H})\}. $$
It  is a Banach space for the norm
$$\|u\|_{MR(p,\mathcal{H})} := \|u\|_{W^1_p(0,\tau; \mathcal{H})}+\|Au\|_{ L_p(0,\tau; \mathcal{H})}.$$

\begin{proof}
 Let $ f \in L_p(0,\tau; \mathcal{H})$ and  $u_0 \in (\mathcal{H},D(A(0))_{1-\frac{1}{p},p}$. By Proposition \ref{prop3-2-1}, there 
exists a unique  $u\in MR(p,\mathcal{H})$ such that 
 \begin{equation*}
 \left\{
\begin{array}{l}
u'(t)+B(0)A(t)u(t) + P(t) u(t) = f(t) \\
 u(0)=u_0.
\end{array}
\right.
\end{equation*}
Hence, for  a given $v\in MR(p,\mathcal{H})$, there exists  a unique $u\in MR(p,\mathcal{H}) $ such that 
\begin{equation} \label{eq3-4-7}
 \left\{
\begin{array}{l}
u'(t)+B(0)A(t)u(t) + P(t) u(t) = f(t)+(B(0)-B(t))A(t)v(t)\\
 u(0)=u_0. 
\end{array}
\right.
\end{equation}
We define 
 \begin{eqnarray*}
 S: MR(p, \mathcal{H})&\rightarrow & MR(p, \mathcal{H})\\ \nonumber
 v & \rightarrow & u.\nonumber 
  \end{eqnarray*}
  For $v_1, v_2 \in MR(p, \mathcal{H})$ we set  $u_1 := Sv_1$ and $u_2 := Sv_2$.  Obviously,
  $u := u_1 -u_2$ satisfies
  \begin{equation*}
 \left\{
\begin{array}{l}
u'(t)+ B(0)A(t)u(t)  + P(t) u = (B(0) - B(t)) A(t)(v_1-v_2) \\
u(0) = 0. 
\end{array}
\right.
\end{equation*}
 Thus, by   Proposition \ref{prop3-2-1}, there exists a constant $C$ such that
\begin{align*}
&\|  u_1-u_2\|_{MR(p, \mathcal{H})}\\
&\le C \| (B(0) - B(\cdot)) A(\cdot)(v_1-v_2) \|_{L_p(0, \tau; \mathcal{H})}\\
&\le C' \sup_{t\in [0, \tau]} \left( \| B(0) - B(t) \|_{\mathcal{L(H)} } \right) \| v_1-v_2\|_{MR(p, \mathcal{H})}.
\end{align*}
By continuity at $0$, for $\epsilon >0$  there exists $t_0 > 0$ such that for $ t \in [0, t_0]$
   $$ \|(B(0)-B(t)) \|_{\mathcal{L}(H)} < \epsilon.$$ 
   Hence for $\tau = t_0$ small enough, the operator $S$ is a contraction on $MR(p,\mathcal{H})$ and so  it has a fixed point $u \in MR(p,\mathcal{H}).$ 
   Clearly, $u $ is a solution of the Cauchy problem \eqref{CPP0} on $[0, t_0]$. The uniqueness of $u$ follows from the uniqueness of the fixed point of $S$. The a priori estimate of Proposition \ref{prop3-2-1} and the fact that $v= u$ on $[0, t_0]$ give
   \begin{equation}\label{apriori1}
  \|  u \|_{MR(p, \mathcal{H})} \le C \left[ \| f \|_{L_p(0, t_0; \mathcal{H})} + \| u_0 \|_{( \mathcal{H}, D(A(0)))_{1-\frac{1}{p}, p}} \right].
  \end{equation}
   
   Next, we divide  the interval $[0,\tau]$ into $ \cup_{i\in{1,....,N}}[t_{i-1},t_i]$ with  $t_i-t_{i-1} $ small enough. On each interval $[t_{i-1}, t_i]$, we search for a solution $u^i$ to \eqref{CPP0}  with initial data $u^i(t_{i-1}) = u^{i-1}(t_{i-1})$. The forgoing arguments prove existence and uniqueness of a solution on $[t_{i-1}, t_i]$ with maximal $L_p$-regularity  provided   $u^{i-1}(t_{i-1}) \in ( \mathcal{H} , D(A(t_{i-1})))_{1-\frac{1}{p},p}$. Once we do this  we glue these solutions and obtain a unique solution $u$ of \eqref{CPP0} with maximal $L_p$-regularity on $[0, \tau]$ for all $u_0 \in  ( \mathcal{H} , D(A(0)))_{1-\frac{1}{p},p}$. Thus, our  task now  is to prove that $u^{i-1}(t_{i-1}) \in ( \mathcal{H} , D(A(t_{i-1})))_{1-\frac{1}{p},p}$. In order to make the notation simpler, we 
work on $[0, \tau]$  (with $\tau$ small enough) instead of $[t_{i-1}, t_i]$ and set $A := A(\tau)$. We have to prove that the solution $u$ to \eqref{CPP0}
satisfies $u(\tau) \in ( \mathcal{H} , D(A))_{1-\frac{1}{p},p}$. This means that (remember we always assume w.l.o.g. that the operators $A(t)$ are invertible, i.e. $\nu = 0$ in [H3])
\begin{equation}\label{eq3-4-8}
t \mapsto A e^{-tA} u(\tau) \in L_p(0, \tau; \mathcal{H}).
\end{equation}
We start with an expression for $u(\tau)$. Set
$$ v(s) := e^{-(\tau-s)A}u(s), \ 0 \le s \le \tau.$$
We have
\begin{equation*}
v'(s) = A e^{-(\tau-s)A} u(s) + e^{-(\tau-s)A} (-B(s)A(s) u(s) - P(s)u(s)+f(s)).
\end{equation*}
Hence
\begin{eqnarray}\label{eq3-4-9}
u(\tau) &=& e^{-\tau A}u_0 + \int_0^\tau e^{-(\tau-s)A}( \mathcal{A}(\tau) - \mathcal{A}(s) ) u(s) ds\\
 &+& \int_0^\tau e^{-(\tau-s)A} \left[ (I-B(s) A(s)u(s) - P(s)u(s)+ f(s) \right] ds.\nonumber
\end{eqnarray}
Since 
$$ \| A e^{-(t+ \tau)A} \|_{\mathcal{L(H)}} \le \frac{C}{t+ \tau}$$
it is clear that  $t \mapsto A e^{-tA}e^{-\tau A} u_0 \in L_p(0, \tau; \mathcal{H})$. \\
For the second term we have
\begin{align*}
& \| A \int_0^\tau e^{-(t+\tau-s)A} ( \mathcal{A}(\tau) - \mathcal{A}(s) ) u(s) ds \|\\
&\le \int_0^\tau \frac{C}{t + \tau-s} \| e^{-\frac{1}{2}(t + \tau-s)A} ( \mathcal{A}(\tau) - \mathcal{A}(s) ) u(s) ds \|\\
&\le C \int_0^\tau \frac{\omega(\tau-s)}{ (t+\tau -s)^{1+ \frac{\gamma}{2}}} \| u(s) \|_{{\mathcal V}_\beta} ds \\
&\le C \int_0^\tau \frac{\tilde{\omega}(t+\tau-s)}{ (t+\tau -s)^{1+ \frac{\gamma}{2}}} \| u(s) \|_{{\mathcal V}_\beta} ds,
\end{align*}
where $\tilde{\omega}(r) = \omega(r)$ for $r \in [0, \tau]$ and $= \omega(\tau)$ for $r > \tau$.
Therefore, using the assumption \eqref{eqmain} on $\omega$ and Young's inequality we obtain
\begin{equation}\label{eq333}
\int_{0}^{\tau}  \| A \int_0^\tau e^{-(t+\tau-s)A} ( \mathcal{A}(\tau) - \mathcal{A}(s) ) u(s) ds \|^{p} dt \le  C' \| u\|^{p}_{L^{p}(0,\tau; \mathcal{V})}.
\end{equation}

Now we consider the last term in \eqref{eq3-4-9}. 
We start with the case $p = 2$. Set 
\begin{equation}\label{eq-g}
 g(s) := (I-B(s) A(s)u(s) - P(s)u(s)+f(s).
 \end{equation}
We have
\begin{align*}
& \left(\int_0^\tau \| A e^{-tA} \int_0^\tau e^{-(\tau -s) A} g(s) ds \|^2 dt \right)^{1/2}\\
& =  \left( \int_0^\tau \| A^{1/2} e^{-tA} \int_0^\tau A^{1/2} e^{-(\tau -s) A} g(s) ds \|^2 dt \right)^{1/2}\\
&\le C \| \int_0^\tau A^{1/2} e^{-(\tau -s) A} g(s) ds \|.
\end{align*}
In the last inequality we use the boundedness of the square function, namely
\begin{equation}\label{square}
\int_0^\infty  \| {A(t)}^{1/2}  e^{-r {A(t)}} x \|^2 dr  \le C \| x \|^2
\end{equation}
for all $x \in \mathcal{H}$. This estimate is a consequence of the fact that $A(t)$ has a bounded holomorphic functional calculus as an accretive operator,  see \cite{CDMY}.\\
We repeat  the previous  argument but since $g$ is not necessarily constant in $s$ we  cannot use directly the square function estimate. We argue by  duality. For $x \in \mathcal{H}$ we have
\begin{align*}
&| (\int_0^\tau A^{1/2} e^{-(\tau -s) A} g(s) ds, x) |  \\
&= |\int_0^\tau (g(s) , {A^*}^{1/2} e^{-(\tau -s) A^*} x) |\\
&\le (\int_0^\tau \| g(s) \|^2 ds)^{1/2} ( \int_0^\tau \|  {A^*}^{1/2} e^{-(\tau -s) A^*} x \|^2 ds)^{1/2}\\
&\le C \| x \| (\int_0^\tau \| g(s) \|^2 ds)^{1/2}. 
\end{align*}
Since this is true for all $x \in \mathcal{H}$, we obtain 
$$\left(\int_0^\tau \| A e^{-tA} \int_0^\tau e^{-(\tau -s) A} g(s) ds \|^2 dt \right)^{1/2} \le C (\int_0^\tau \| g(s) \|^2 ds)^{1/2}.$$
We define the operator $T$ by
$$ Tg(t) = \int_0^\tau A e^{-(\tau + t -s) A} g(s) ds.$$
We have proved that $T : L_2(0, \tau; \mathcal{H}) \to L_2(0, \tau ; \mathcal{H}) $ is bounded. We extend this operator to 
$L_p(0, \tau ; \mathcal{H})$ for all $p \in (1, \infty)$. Indeed, note that $T$ is a singular integral  operator with kernel
$$ K(t,s) = A e^{-(\tau + t -s) A}$$
and we   use  H\"ormander's integral condition for $K(t,s)$ and $K(s,t)$ (see, e.g. \cite{RRT86}, Theorems III 1.2 and III 1.3). A similar argument was used in \cite{HO15}. We have to prove that 
\begin{equation}\label{Horm1}
 \int_{|t-s|\ge 2|s'-s|} \| K(t,s) - K(t,s')\|_{\mathcal{L(H)}} dt \le C
\end{equation}
for some constant $C$ independent of $s, s' \in (0, \tau)$. 

Assume for example that $s \le s'$. Since the semigroup generated by $-A$ is bounded holomorphic  we have 
for some constant $C$
\begin{eqnarray*}
&&\int_{|t-s|\ge 2|s'-s|} \| K(t,s) - K(t,s')\|_{\mathcal{L(H)}}  dt \\
&=& \int_{2s'-s}^\tau \| A e^{-(\tau+ t-s)A} - A e^{-( \tau + t-s')A} \|_{\mathcal{L(H)}} \\
&=&  \int_{2s'-s}^\tau \|  \int_s^{s'} A^2 e^{-(\tau + t-r) A} dr \|_{\mathcal{L(H)}}  dt\\
&\le& C \int_{2s'-s}^\tau  \int_s^{s'} \frac{1}{(\tau+ t-r)^2} dr dt\\
&=& C \int_{2s'-s}^\tau  \left[ \frac{1}{\tau+ t-s'} - \frac{1}{\tau+ t-s} \right] dt\\
&=& C \left[ \log \frac{\tau+ t-s}{\tau+ t-s'} \right]_{t = 2s'-s}^{t= \tau} \le C \log 2.
\end{eqnarray*}
This proves \eqref{Horm1}. The same arguments apply for the kernel of the adjoint $T^*$. Hence 
$$ T: L_p(0, \tau; \mathcal{H}) \to L_p(0, \tau; \mathcal{H})$$
is a bounded operator for all $p \in (1, \infty)$. We obtain
\begin{equation}\label{eq333-1}
\left(\int_0^\tau \| A e^{-tA} \int_0^\tau e^{-(\tau -s) A} g(s) ds \|^p dt \right)^{1/p}\le C (\int_0^\tau \| g(s) \|^p ds)^{1/p}.
\end{equation}
Note that for $f \in L_p(0, \tau; \mathcal{H})$ we have $A(\cdot) u(\cdot) \in  L_p(0, \tau; \mathcal{H})$ since we have proved maximal 
$L_p$-regularity for small $\tau$. Hence $g \in  L_p(0, \tau; \mathcal{H})$ (remember that $g$ is given by \eqref{eq-g}).  This finishes the proof of 
$u(\tau) \in ( \mathcal{H} , D(A(\tau))_{1-\frac{1}{p},p}$ and we obtain maximal $L_p$-regularity of \eqref{CPP0} on 
$[0, \tau]$ for every $\tau > 0$. The uniqueness of the solution on $[0, \tau]$ follows from the uniqueness on each small sub-interval 
$[t_i, t_{i+1}]$. 

It remains to prove the a priori estimate \eqref{apriori}. 
On each small sub-interval $[t_{i};t_{i+1}]$ we have the a priori estimate \eqref{apriori1}. That is
\begin{align*}
&\| u \|_{W^1_p(t_{i},t_{i+1}; \mathcal{H})} +\| A(\cdot)  u(\cdot) \|_{L_p(t_{i},t_{i+1}; \mathcal{H})}\\
& \leq C \left[\|f\|_{L_p(t_{i},t_{i+1}; \mathcal{H})}+\|u(t_{i})\|_{( \mathcal{H};D(A(t_{i})))_{1-\frac{1}{p};p}}\right].
\end{align*}
Using again the expression  \eqref{eq3-4-9} and the estimates \eqref{eq333} and \eqref{eq333-1}  we obtain
\begin{equation}\label{eq333-2}
\|u (t_{i}) \|_{( \mathcal{H};D(A(t_{i})))_{1-\frac{1}{p};p}} \leq C \left[\|f\|_{L_p(0,t_{i}; \mathcal{H})}+ \|u_{0}\|_{( \mathcal{H}, D(A(0)))_{1-\frac{1}{p},p}}+              
 \|u\|_{L_p(0,t_{i}; \mathcal{V})}  \right].
 \end{equation}
 Remember that we can replace $A(t)$ by  $A(t)+ cI$ for a given constant  $c \in (0,+\infty)$. Since
 $$\| (A(t)+cI)^{-1} \|_{\mathcal{L(H,V)}} \leq \frac{c_{0}}{(c +1)^{\frac{1}{2}}} $$
 it follows that
 \begin{align*} 
 \| u(t) \|_{\mathcal{V}} &\leq \| (A(t)+cI)^{-1} \|_{\mathcal{L(H,V)}} \|(A(t)+ cI) u(t)\|_{\mathcal{H}} \\
  &\leq \frac{c_{0}}{(c +1)^{\frac{1}{2}}} \|(A(t)+ cI) u(t)\|_{\mathcal{H}}. 
  \end{align*}
Summing over $i$ in \eqref{eq333-2} and taking $c $ large enough we see that for some constant $C_{1}$
\begin{align*}
&\| u \|_{W^1_p(0,\tau; \mathcal{H})} +\| A(\cdot)  u(\cdot) \|_{L_p(0,\tau; \mathcal{H})}\\
& \leq C_{1} \left[\|f\|_{L_p(0,\tau; \mathcal{H})}+\|u_{0}\|_{( \mathcal{H}, D(A(0)))_{1-\frac{1}{p},p}}\right].
\end{align*}
This proves the desired a priori estimate and finishes the proof of the theorem.
\end{proof}

\section{Further regularity results}\label{sec-reg}
We continue our investigations on the solution of the problem \eqref{CPP}. We work with  the same assumptions as in Theorem \ref{thmain}.
For $f \in L_2(0, \tau; \mathcal{H})$ the solution $u \in W^1_2(0, \tau; \mathcal{H})$ and by the Sobolev embedding 
$ u \in C([0, \tau]; \mathcal{H})$. It is interesting to know  whether  $u$ is also continuous for the norm of $\mathcal{V}$. This is indeed the case if the forms $\fra(t)$ are symmetric (or perturbations of symmetric forms) and $t \mapsto \fra(t, x,y)$ is Lipschitz continuous on $[0, \tau]$ for all $x, y \in \mathcal{V}$. This is proved in \cite{ADLO}. Continuity in $\mathcal{V}$ was also proved in  \cite{AM} for the unperturbed problem (i.e., without multiplicative and additive perturbations) when  $\gamma$ in \eqref{eqeq0} is $< 1$. This is a rather restrictive condition but turns out to be satisfied in some cases such as time-dependent  Robin boundary conditions. Here we make no restriction on $\gamma$ and $\beta$ and we assume less regularity for  $t \mapsto \fra(t,x,y)$ than what was previously known. \\
The continuity of the solution with respect to the norm of  $\mathcal{V}$ is used in \cite{ADLO} in applications to some semi-linear PDE's. In this section, we look at again this question in the setting of Theorem \ref{thmain} in which we assume less regularity (than Lipschitz continuous) on $t \mapsto \fra(t, g,h)$. 

In the statements below we shall need the following square root property (called Kato's square root property)
\begin{equation}\label{reg1}
 D(A(t)^{1/2}) = \mathcal{V} \ \text{and}  \  c_1 \| A(t)^{1/2} v \| \le \| v \|_{\mathcal{V}} \le c_2 \| A(t)^{1/2} v \|
 \end{equation}
for all $v \in \mathcal{V}$ and  $ t \in [0, \tau]$,  where the  positive constants $c_1$ and $c_2$  are independent of $t$. Note that this assumption is always true for symmetric forms when 
  $\nu = 0$ in [H3]. 

We start with the following lemma which will be used later.

\begin{lem}\label{lemreg11}
Suppose \eqref{reg1}. Then  for all $f\in L^{2}(0,t; \mathcal{H})$, $ 0 \le s\le t \le \tau$,
$$\|\int_{s}^{t} { e^{-(t-r)A(t)}f(r)}dr \|_{\mathcal{V}} \leq C \|f\|_{L^{2}(s,t; \mathcal{H})}.$$
\end{lem}
\begin{proof}
By \eqref{reg1},
\begin{align*}
\|\int_{s}^{t} { e^{-(t-r)A(t)}f(r)}dr \|_{\mathcal{V}} &\leq c_2 \|\int_{0}^{t} A(t)^{1/2}  e^{-(t-r)A(t)}f(r) dr \|\\
&=  c_2  \sup_{\|x\|=1} | \int_s^t (f(r) ,   {A(t)^*}^{1/2} e^{-(t-r){A(t)^*}} x ) dr | \\
&\leq c_2 \sup_{\|x\|=1} \left(\int_{s}^{t} \| {A(t)^*}^{1/2}  e^{-(t-r){A(t)^*}} x \|^2 dr \right)^{\frac{1}{2}} \times\\
& \hspace{1cm}  \| f \|_{L_2(s,t; \mathcal{H})}  \\
&\leq C \|f\|_{L^{2}(s,t; \mathcal{H})}.
\end{align*}
Note that in the last inequality we use again  the square function  estimate for $A(t)^*$ (see \eqref{square}). This proves the lemma. 
\end{proof}
In the next result we prove continuity of the solution to \eqref{CPP} as a function with values in $\mathcal{V}$. Note that 
if $D(A(0)^{1/2}) = \mathcal{V}$, then
$$ (D(A(0), \mathcal{H})_{\frac{1}{2}, 2} = D(A(0)^{1/2}) = \mathcal{V}.$$
\begin{teo}\label{contiV}
Suppose \eqref{reg1}  and that the assumptions of Theorem \ref{thmain} are satisfied. Suppose also that 
 $\omega(t) \le c t^{\varepsilon}$ for some $\varepsilon > 0$.   Let $f \in L_2(0, \tau; \mathcal{H})$ and  $u_0 \in \mathcal{V}$. Then the solution $u$ to  the problem \eqref{CPP} satisfies $u\in C([0,\tau]; \mathcal{V}).$
\end{teo}
\begin{proof}
We fix $s$ and $t$ in $[0, \tau]$ such that $s<t$. We first derive a formula similar to  \eqref{eq3-4-9}.  Define 
$v(r) :=e^{-(t-r)A(t)} u(r)$ 
for  $r\in [s,t]$.  After derivation and integration from $s$ to $t$ we obtain 
\begin{align*}
u(t)&=e^{-(t-s)A(t)}u(s)+ \int_{s}^{t} { e^{-(t-r)A(t)}( \mathcal{A}(t)-B(r)\mathcal{A}(r))u(r)dr}\\   &+\int_{s}^{t} {e^{-(t-r)A(t)} [-P(r)u(r)+f(r)]}dr.
\end{align*} 
Hence
 \begin{align}\label{4444}
   u(t)&=e^{-(t-s)A(t)}u(s)+ \int_{s}^{t} { e^{-(t-r)A(t)}(\mathcal{A}(t)-\mathcal{A}(r))u(r)dr} \nonumber \\  
   &+ \int_{s}^{t} { e^{-(t-r)A(t)}[(-B(r)+I)A(r)u(r)-P(r)u(r)+f(r)]}dr,
   \end{align}
 and so 
 \begin{align*}
 \|u(t)-u(s)\|_{\mathcal{V}} 
 & \leq \|\int_{s}^{t} { e^{-(t-r)A(t)}(\mathcal{A}(t)-\mathcal{A}(r))u(r)dr}\|_{\mathcal{V}} \\ 
 &\hspace{-.5cm}+ \| \int_{s}^{t} { e^{-(t-r)A(t)}[(I-B(r))A(r)u(r)-P(r)u(r)+f(r)]}dr\|_{\mathcal{V}}\\
&\hspace{-.5cm} +\|[e^{-(t-s)A(t)}-I]u(s)\|_{\mathcal{V}}.
\end{align*}
We estimate each term in the RHS and obtain 
 \begin{align*}
 \|u(t)-u(s)\|_{\mathcal{V}} & \leq C(  \int_{s}^{t}  \frac{\omega(t-r)}{t-r}\|u(r)\|_{\mathcal{V}}dr\\
 & + \|f\|_{L^{2}(s,t; \mathcal{H})} +\|Au\|_{L^{2}(s,t; \mathcal{H})}+\|u\|_{L^{2}(s,t; \mathcal{H})}) \\
 &+\|[e^{-(t-s)A(t)}-I]u(s)\|_{\mathcal{V}}\\
 &\leq  C ( \int_{0}^{t-s}\frac{\omega(r)}{r} dr \|u\|_{L^{\infty}(0,\tau; \mathcal{V})}+\|f\|_{L^{2}(s,t; \mathcal{H})}\\
 &+\|Au\|_{L^{2}(s,t; \mathcal{H})}+\|u\|_{L^{2}(s,t; \mathcal{H})}) \\
 &+\|[e^{-(t-s)A(t)}-I]u(s)\|_{\mathcal{V}}.
 \end{align*}
Note that $MR(2, \mathcal{H})$ is continuously embedded into $L_\infty(0, \tau; \mathcal{V})$ by Proposition 4.5 in \cite{MO}. This proposition is proved for  forms which are symmetric  but  it  remains true under the assumption \eqref{reg1}. Thus, by maximal regularity result (Theorem \ref{thmain}), $u \in MR(2, \mathcal{H}) $ and hence $u \in L_\infty(0, \tau; \mathcal{V})$.  

Next, since $\omega$ satisfies the assumptions of Theorem \ref{thmain}, $\int_0^{\tau} \frac{\omega(r)}{r} dr < \infty$ and since 
$u, Au \in L_2(0, \tau; \mathcal{H})$ we see that the first four terms in the RHS converge to $0$ as $t \to s$ (or as $s \to t$).  It remains to proves that 
\begin{equation}\label{44conv}
\|[e^{-(t-s)A(t)}-I]u(s)\|_{\mathcal{V}} \to 0 \ \text{as} \ t \to s \ (\text{or as } \ s \to t).
\end{equation}
We first prove \eqref{44conv} when $t \to s$ (for fixed $s$). We write 
 \begin{align*}
\|[e^{-(t-s)A(t)}-I]u(s)\|_{\mathcal{V}}&\leq \|[e^{-(t-s)A(t)}-e^{-(t-s)A(s)}]u(s)\|_{\mathcal{V}} \\
&+\|[e^{-(t-s)A(s)}-I]u(s)\|_{\mathcal{V}}. 
\end{align*} 
Using Proposition \ref{prop3-1} and the functional calculus (on the sector $\Sigma(\theta)$ for appropriate $\theta \in (0, \frac{\pi}{2})$) we estimate the first terms as follows.
\begin{align*}
&\|[e^{-(t-s)A(t)}-e^{-(t-s)A(s)}]u(s)\|_{\mathcal{V}}\\ &= \frac{1}{2\pi} \|\int_{\Gamma} e^{-(t-s)\lambda}[(\lambda-A(t))^{-1}(\mathcal{A}(t)-\mathcal{A}(s))(\lambda-A(s))^{-1}] u(s)d\lambda\|_{\mathcal{V}}\\
&\leq c    \omega(t-s)  \int_{0}^{\infty} e^{-(t-s)|\lambda| \cos\theta} (1+|\lambda|)^{-1}d|\lambda|\|u(s)\|_{\mathcal{V}}  \\
&\leq c' \frac{\omega(t-s)}{(t-s)^{\varepsilon'} } \|u(s)\|_{\mathcal{V}}. 
\end{align*} 
Here we use $u(s)\in D(A(s)^{\frac{1}{2}}) $ (see the proof of Theorem \ref{thmain}) and \eqref{reg1}. Now the fact that 
$\omega(t) \le c t^{\varepsilon}$ for some $\varepsilon > 0$ and the strong continuity of the semigroup $e^{-tA(s)}$ on $\mathcal{V}$ imply that   
$\|[e^{-(t-s)A(t)}-I]u(s)\|_{\mathcal{V}}  \to 0$ as $t \to s$. This proves that $u$ is right continuous for the norm of $\mathcal{V}$. 

It remains to prove left continuity of $u$. We need a formula similar to \eqref{4444} but with $u(s)$
expressed in terms of $u(t)$. Fix $0 \le s < t \le \tau$ and set $v(r) :=e^{-(r-s)A(s)}u(r)$ for $r \in [s, t]$.  
Then 
$$v'(r) = -e^{-(r-s)A(s)}(A(s)+B(r)A(r)+P(r))u(r))+e^{-(r-s)A(s)}f(r),$$
and hence
\begin{align}\label{spetit}
u(s) &=e^{-(t-s)A(s)}u(t)  +\int_{s}^{t}e^{-(r-s)A(s)}(\mathcal{A}(s)+B(r)A(r))u(r) dr \nonumber
\\ &-\int_{s}^{t}e^{-(r-s)A(s)}[f(r)-P(r)u(r)] dr. 
\end{align}
Therefore 
\begin{align*}
u(s)-u(t)&=[e^{-(t-s)A(s)}u(t)-u(t)]+\int_{s}^{t}e^{-(r-s)A(s)}(\mathcal{A}(s)- A(r))u(r) dr \\
&+\int_{s}^{t}e^{-(r-s)A(s)}((B(r)+I) A(r)u(r)) dr  \\
&-\int_{s}^{t}e^{-(r-s)A(s)}[f(r)-P(r)u(r)] dr \\
&=: I_{1}(s,t)+I_{2}(s,t)+I_{3}(s,t)+I_{4}(s,t).
\end{align*} 
By Lemma \ref{lemreg11},
$$\|I_{4}(s,t)\|_{\mathcal{V}} \leq C [\|u\|_{L^{2}(s,t; \mathcal{H})}  + \|f\|_{L^{2}(s,t ; \mathcal{H})} ].$$ 
By Lemma \ref{lemreg11}
$$\|I_{3}(s,t)\|_{\mathcal{V}} \leq C \|A(\cdot) u (\cdot) \|_{L^{2}(s, t; \mathcal{H})}.$$ 
For $I_{2}(s,t) $ we have  immediately, 
$$\|I_{2}(s,t)\|_{\mathcal{V}} \leq C \int_{s}^{t} \frac{w(r-s)}{r-s} dr \|u\|_{L^{\infty}(0;\tau, \mathcal{V})}.$$
For $I_{1}(s,t) $ we proceed as before. We write
\begin{align*}
[e^{-(t-s)A(s)}u(t)-u(t)]&=[e^{-(t-s)A(s)}u(t)-e^{-(t-s)A(t)}u(t)]\\
&+[e^{-(t-s)A(t)}u(t)-u(t)]
\end{align*}
We use again the functional calculus as above to obtain 
$$\|[e^{-(t-s)A(s)}u(t)-e^{-(t-s)A(t)}u(t)]\|_{\mathcal{V}} \leq c \frac{\omega(t-s)}{(t-s)^{\varepsilon'}} \|u(t)\|_{\mathcal{V}}.$$
The remaining term $\|e^{-(t-s)A(t)}u(t)-u(t)\|_{\mathcal{V}}$ converges to $0$ as $s \to t$ by strong continuity of the semigroup on $\mathcal{V}$. 
We have proved that $u$ is left continuous in $\mathcal{V}$ and finally  $u \in C([0,\tau]; \mathcal{V})$. 
 \end{proof}
 
 \begin{prop}\label{prop44}
  Suppose that the assumptions of the previous theorem are satisfied. Let $f \in L_2(0, \tau; \mathcal{H})$, $ u_0 \in \mathcal{V}$ and $u$ be the solution  of \eqref{CPP}. Then 
$$s \rightarrow A(s)^{\frac{1}{2}}u(s)\in C([0,\tau]; \mathcal{H}).$$
\end{prop}
\begin{proof}
We use again \eqref{4444} and write
\begin{align*}
&A(t)^{\frac{1}{2}}u(t)-A(s)^{\frac{1}{2}}u(s)\\
&=A(t)^{\frac{1}{2}}e^{-(t-s)A(t)}u(s)-A(s)^{\frac{1}{2}}u(s)\\
&+ A(t)^{\frac{1}{2}} \int_{s}^{t} { e^{-(t-r)A(t)}(\mathcal{A}(t)-\mathcal{A}(r))u(r)dr}\\
&+A(t)^{\frac{1}{2}} \int_{s}^{t} { e^{-(t-r)A(t)}[(-B(r)+I)A(r)u(r)-P(r)u(r)+f(r)]}dr.
\end{align*}
By \eqref{reg1}, the norms in $\mathcal{H}$ of the last two terms are equivalent to the norms in $\mathcal{V}$ of the same terms but without 
$A(t)^{\frac{1}{2}}$. We have seen in the proof of Theorem  \ref{contiV} that these norms in $\mathcal{V}$ converge to $0$ as $t \to s$ or as $s\to t$. It remains to consider the  term $A(t)^{\frac{1}{2}}e^{-(t-s)A(t)}u(s)-A(s)^{\frac{1}{2}}u(s).$
We use again  the functional calculus to write
\begin{align*}
&A(t)^{\frac{1}{2}}e^{-(t-s)A(t)}u(s)-A(s)^{\frac{1}{2}}u(s)\\
&=A(t)^{\frac{1}{2}}e^{-(t-s)A(t)}u(s)-A(s)^{\frac{1}{2}}e^{-(t-s)A(s)}u(s)\\ 
&\hspace{.5cm} +A(s)^{\frac{1}{2}}e^{-(t-s)A(s)}u(s)-A(s)^{\frac{1}{2}}u(s)\\
&=\frac{1}{2 \pi i} \int_{\Gamma} \lambda^{\frac{1}{2}}e^{-(t-s)\lambda}[(\lambda-A(t))^{-1}-(\lambda-A(s))^{-1}]u(s) d\lambda\\
&\hspace{.5cm} + A(s)^{\frac{1}{2}}e^{-(t-s)A(s)}u(s)-A(s)^{\frac{1}{2}}u(s).
\end{align*}
By the resolvent equation
$$(\lambda-A(t))^{-1}-(\lambda-A(s))^{-1} = (\lambda -A(t))^{-1} (\mathcal{A}(s) - \mathcal{A}(t)) (\lambda -A(s))^{-1}$$
and Proposition \ref{prop3-1} we have 
\begin{align*}
&\| (\lambda -A(t))^{-1} - (\lambda -A(s))^{-1} \| \\
& \le  \|(\lambda-A(t))^{-1} \|_{ \mathcal{L(V',H)}} \| (\mathcal{A}(s) - \mathcal{A}(t)) \|_\mathcal{L(V,V')} \|(\lambda-A(s))^{-1} \|_{ \mathcal{L(V)}}\\
&\le C | \lambda |^{-1/2} \omega(|t-s|) \frac{1}{ 1 + |\lambda|}.
\end{align*}
Therefore for any $\varepsilon > 0$, 
\begin{align*}
& \|A(t)^{\frac{1}{2}}e^{-(t-s)A(t)}u(s)-A(s)^{\frac{1}{2}}e^{-(t-s)A(s)}u(s)\|\\
 & \leq C \omega(|t-s|) \int_{0}^{\infty} \frac{1}{1 + r} e^{-(t-s)r}dr \|u(s)\|_\mathcal{V}\\
& \leq C_\varepsilon  \frac{\omega(|t-s|)}{ |t-s|^{\varepsilon'}}  \|u(s)\|_\mathcal{V}.
\end{align*}
Remember  that $u \in L_\infty(0, \tau ; \mathcal{V})$ by Theorem \ref{contiV}.  Using the assumption on $\omega$,
the latest term converges to $0$ as $|t-s| \to 0$.  The  term $ A(s)^{\frac{1}{2}}e^{-(t-s)A(s)}u(s)-A(s)^{\frac{1}{2}}u(s)$ converges to $0$ as $t  \to s$ by the strong continuity of the semigroup. 
This proves the right continuity of $s \mapsto A(s)^{\frac{1}{2}}u(s)$. The left continuity is proved similarly, we  use 
\eqref{spetit} instead of \eqref{4444}. 
\end{proof}


\section{Right perturbations-Maximal regularity}\label{sec4}
 Let $B(t)$ and $P(t)$ ($t\in[0,\tau]$) be bounded operators on $\mathcal{H}$. We 
 investigate the maximal $L_p$-regularity property for right multiplicative perturbations $A(t) B(t)$. As mentioned in the introduction, this problem has been considered in \cite{JL15} and was motivated there by several applications. We will extend the results from \cite{JL15} in the sense that we  require much less regularity for  $t \mapsto \fra(t)$.
 
Let $\fra(t)$ be a family of sesquilinear forms satisfying again [H1]-[H3] and denote as before $A(t)$ the corresponding associate operators. Under the assumptions of Theorem \ref{thmain}, for each $t$, the operator $-B(t)A(t)$ generates a holomorphic semigroup on $\mathcal{H}$. The same is also true for $-B(t)^* A(t)^*$ since the adjoint operators  $B(t)^*$ and $A(t)^*$ satisfy the same properties as $B(t)$ and $A(t)$. Hence by duality,  $-A(t)B(t)$ generates a holomorphic semigroup on $\mathcal{H}$. Here, the domain of $A(t)B(t)$ is given by 
$$D(A(t)B(t))=\{x \in \mathcal{H},  B(t)x \in D(A(t))\}.$$
For right perturbations, we say that the Cauchy problem 
$$u'(t) + A(t)B(t) u(t) + P(t) u(t) = f(t), \ \ u(0) = u_0$$
has maximal $L_p$-regularity if for every $f \in L_p(0, \tau, \mathcal{H})$ there exists a unique 
$u \in W^1_p(0, \tau; \mathcal{H})$, $B(t)u(t) \in D(A(t))$ a.e. and $u$ satisfies the Cauchy problem in the $L_p$-sense.  

Our main result in this section is the following. 

\begin{teo}\label{thmright}
Let  $(\fra(t))_t$ satisfy  [H1]-[H3] and  $B(t)$ and $P(t)$ be bounded operators  satisfying  \eqref{eq3-4-1}-\eqref{eq3-4-3} and 
\eqref{eq3-4-5}.  We suppose that $t \mapsto B(t) $ is Lipschitz  and  $t \mapsto P(t)$ is  strongly measurable.  
Suppose that for some $\beta, \gamma \in [0, 1]$
\begin{equation*}
|\fra(t, u,v)- \fra(s, u,v)| \leq \omega(|t-s|)\|u\|_{\mathcal{V}_{\beta}}  \|v\|_{\mathcal{V}_{\gamma}}, \ u, v \in \mathcal{V}
\end{equation*}
where $\omega: [0,\tau] \to  [0,\infty)$ is a non-decreasing function such that :
$$\int_0^{\tau} \frac{\omega(t)}{t^{1+\frac{\gamma}{2}}}dt <\infty.$$
Then the Cauchy problem 
\begin{equation}\label{CPright}
u'(t) + A(t)B(t)u(t)  + P(t)u(t) = f(t),  \ \ u(0) = 0
\end{equation}
 has maximal $L_p$-regularity in $\mathcal{H}$ for all $p \in (1, \infty)$. \\
If in addition, 
\begin{equation}\label{eq3-4-6} 
\int_{0}^{\tau}\frac{\omega(t)^{p}}{t^{\frac{1}{2}(\beta+p \gamma)}}dt<\infty
\end{equation}
then 
\begin{equation}\label{CPPright}
u'(t) + A(t)B(t)u(t)  + P(t)u(t) = f(t), \ \ u(0) = u_0
\end{equation}
 has maximal $L_p$ regularity in $\mathcal{H}$ provided $u_0 \in B(0)^{-1}(\mathcal{H}, D(A(0)))_{1-\frac{1}{p},p}$. 
 Moreover there exists a positive constant C such that :
\begin{eqnarray}\label{aprioriR}
&&\| u \|_{W^1_p(0,\tau; \mathcal{H})} +\|A(\cdot)  B(\cdot) u(\cdot) \|_{L_p(0,\tau; \mathcal{H})}\\ \nonumber
&& \  \ \leq C \left[\|f\|_{L_p(0,\tau; \mathcal{H})}+\|B(0)u_0\|_{( \mathcal{H}, D(A(0)))_{1-\frac{1}{p}, p}} \right].
\end{eqnarray}
\end{teo}

We start with the following lemma. 
\begin{lem}\label{lem4-1}
Under the above assumptions on $B(t)$, the mapp $t \mapsto B(t)^{-1}x$ is differentiable on $(0, \tau)$ with values in 
$\mathcal{L(H)}$ and 
$$\frac{d}{dt} B(t)^{-1}x = -B(t)^{-1} B'(t) B(t)^{-1} x$$
for all $x \in \mathcal{H}$. 
\end{lem} 
\begin{proof}
We write
\begin{equation}\label{eq4-1}
B(t+h)^{-1} - B(t)^{-1} = - B(t+h)^{-1} ( B(t+h) - B(t) ) B(t)^{-1}
\end{equation}
and since  $B(t+h)^{-1}$ has norm bounded with respect to $h$ it follows that  $B(t+h)^{-1}$ converges uniformly to $B(t)^{-1}$. Using this and the fact that $t \mapsto B(t)x$ is Lipschitz we obtain the lemma. \end{proof}

\begin{proof}[Proof of Theorem \ref{thmright}] 
Let $f \in L_p(0, \tau; \mathcal{H})$ and initial data  $u_0$ such that $u_0 \in B(0)^{-1}(\mathcal{H}, D(A(0)))_{1-\frac{1}{p},p}$. 
 We consider the Cauchy problem with left multiplicative perturbations
\begin{equation}\label{eq4-2}
 \left\{
\begin{array}{l}
v'(t)+B(t)A(t)v(t)- B'(t)B(t)^{-1}v(t) + B(t) P(t) B(t)^{-1}v(t) = B(t)f(t) \\ 
v(0)=B(0)u_0.
\end{array}
\right.
\end{equation}
Note that $B(\cdot) f(.) \in L_p(0,\tau; \mathcal{H})$  and $B(0) u_0 \in (\mathcal{H}, D(A(0)))_{1-\frac{1}{p},p}$. Note also that $t \mapsto - B'(t)B(t)^{-1} + B(t) P(t) $ is strongly measurable  with values in $\mathcal{L(H)}$. Thus, we can apply Theorem \ref{thmain}. We obtain existence and uniqueness of  
$v \in W^{1}_p(0, \tau; \mathcal{H})$ such that 
$v(t) \in D(A(t))$ for a.e. $t \in [0, \tau]$ which satisfies \eqref{eq4-2}.
We set  $u(t) := B(t)^{-1}v(t)$. Using Lemma \ref{lem4-1} we check easily that  $u \in W^1_p(0, \tau; \mathcal{H})$, $B(t)u(t) \in D(A(t))$ for a.e. $t$ and it  is the unique solution of \eqref{CPright}. Finally, \eqref{aprioriR} follows immediately from the a priori estimate of Theorem \ref{thmain}. 
\end{proof}

Note that we may consider both left and right multiplicative perturbations at the same time.  Let $B_0(t)$ and $B_1(t)$ be bounded operators satisfying the same assumptions  \eqref{eq3-4-1} and  \eqref{eq3-4-2}. We assume that $t \mapsto B_0(t)$ is continuous  and $t \mapsto B_1(t)$ is Lipschitz continuous on $[0, \tau]$. We assume that the forms $\fra(t)$ and $P(t)$ are as in Theorem \ref{thmright}. We consider the Cauchy problem
\begin{equation}\label{leftright}
u'(t) + B_0(t) A(t) B_1(t) u(t) + P(t) u(t) = f(t), \ \ u(0) = u_0.
\end{equation}
Then the maximal $L_p$-regularity results  of Theorem \ref{thmright} hold for \eqref{leftright} for initial data
$u_0 \in B_1(0)^{-1}(\mathcal{H}, D(A(0)))_{1-\frac{1}{p},p}$. The proof is very similar to the previous one. We consider the Cauchy problem with left perturbations
\begin{equation}\label{eq4-2-11}
 \left\{
\begin{array}{l}
v'(t)+B_1(t)B_0(t)A(t)v(t)- B_1'(t)B_1(t)^{-1}v(t) + B_1(t) P(t)B_1(t)^{-1} v(t)\\
\hspace{.8cm} =  B_1(t)f(t) \\ 
v(0)=B_1(0)u_0.
\end{array}
\right.
\end{equation}
We obtain the maximal $L_p$-regularity for \eqref{eq4-2-11} by Theorem \ref{thmain} and set as above $u(t) = B_1(t)^{-1} v(t)$.

\end{document}